\documentclass{amsart}
\usepackage{amssymb,amsmath,amsthm}

\title{Statistical Properties of Martin-L\"of Random Sequences}
\usepackage[colorlinks=true,linkcolor=blue,citecolor=red]{hyperref}
\usepackage[capitalize]{cleveref}

\author{Matthew Pancia}
\address{Department of Mathematics, University of Texas at Austin,
Austin TX, 78712}
\email{mpancia@math.utexas.edu}
\urladdr{www.ma.utexas.edu/users/mpancia}
\thanks{This paper is the result of work done during the MASS Program at Penn State University. The author would like to thank Stephen G. Simpson for his guidance and support with both the research and the writing of this paper.}


\newtheorem{theorem}{Theorem}[section]
\newtheorem{lemma}{Lemma}[section]
\newtheorem*{corollary}{Corollary}     
\theoremstyle{remark}
\newtheorem*{remark}{Remark}
\newtheorem*{note}{Note}
\theoremstyle{definition}

\makeatletter
\let\c@lem=\c@thm
\makeatother

\newcommand{\nn}{\mathbb{N}}
\newcommand{\rr}{\mathbb{R}}
\newcommand{\pz}[1]{\Pi^{0}_{#1}}
\newcommand{\sz}[1]{\Sigma^{0}_{#1}}

\begin{document}

\begin{abstract}    
We study the statistical properties of random numbers under the Martin-L\"of definition of randomness, proving that random numbers obey analogues of Strong Law of Large Numbers, the Law of the Iterated Logarithm, and that they are normal. We also show that weakly (1-)random numbers do not share these properties.
\end{abstract}

\maketitle


\section{Introduction} 

\label{sec:introduction}

The study of random phenomena is an important and interesting field of mathematical research. Despite the pervasiveness of such phenomena throughout mathematics, however, there remains to be a definitive and wholly convincing definition of precisely what randomness \emph{is}. This is not for lack of trying - there are several competing definitions of randomness, but the question remains: How do we decide which one is the ``right'' definition?
One way to gauge the efficacy of these definitions is to look at the properties of random objects that arise when different definitions are used. We concern ourselves in particular with the statistical behavior of random sequences of $0$'s and $1$'s under different choices of what it means to be random. We would, of course, expect a proper definition of randomness to yield the sort of statistical properties in the digits of random sequences that coincide with what we normally consider random behavior. That is, the values of a random sequence should obey the Strong Law of Large Numbers, the Law of the Iterated Logarithm, and all of the other properties that are enjoyed by independent, identically distributed (i.i.d.) random variables acting on most elements of the space of infinite sequences of (fair) coin tosses. The aim of this paper is to explore and prove these properties for Martin-L\"of random sequences and to do some basic comparison with other definitions of randomness. Similar results can be found in previously published papers such as \cite{davie} and \cite{vovk1988law}, but we hope that the reader will find our exposition to be more unified and accessible.

\subsection{Acknowledgements}
This paper is the result of work done during the MASS Program at Penn State University. The author would like to thank Stephen G. Simpson for his guidance and support with both the research and the writing of this paper.
\section{Preliminaries and Notation}
\label{sec:preliminaries_and_notation}
We now set notation and provide backround for the rest of the paper. Readers familiar with this material may skip to \cref{sec:statistical_properties}, where we prove our main results.
\subsection{Computability Theory Basics}
\label{sub:computability_theory_basics}

We let $2^\nn$ denote the \emph{Cantor space}, i.e. the set of infinite sequences of 0's and 1's or total functions $X : \nn \rightarrow \{ 0,1 \}$. We will use $X(i)$ to refer to the ($i+1$)-st value of the sequence $X$ (starting the indexing at 0) or, treating $X$ as a binary expansion of a number in $[0,1]$, its $(i+1)$-st digit. $S_n$ will generally refer to the sum of the first $n$ digits of $X \in 2^\nn$. For a subset $V$ of $2^\nn$, we will denote its complement $2^\nn \setminus V$ as $\overline{V}$.

$2^{< \nn}$ will be the set of \emph{bitstrings}, or finite sequences of 0's and 1's. We will denote the \emph{empty string} by $<>$. If $\sigma, \tau$ are bitstrings, we denote their concatenation by $\sigma ^\smallfrown \tau$, and the length of $\sigma$ by $|\sigma|$.

We denote the bitstring generated by restricting $X \in 2^\nn$ to an initial segment of length $n$ by $X \upharpoonright n $. Let $N_\sigma$ be the \emph{neighborhood} generated by a bitstring $\sigma$, defined as $\{ X \in 2^\nn : X \upharpoonright |\sigma| = \sigma \}$. That is, $N_\sigma$ consists of all $X$ that have $\sigma$ as an initial segment. We let $\mu$ denote the unique Borel measure on $2^\nn$ such that $\mu(N_\sigma) = \frac{1}{2^{|\sigma|}}.$ Note that $2^\nn$ with this probability measure is the same as the space of infinite fair coin tosses, letting 1 denote heads and 0 denote tails.

We define $\varphi_{e}(x)$ to be the output of the program with G\"odel number $e$ running with input $x$ (if it halts).

Letting $X$ be an element of $2^\nn$, we define $\varphi_{e}^{X}(x)$ as above, except $X$ is used as an oracle in the computation.

The following is a standard result, proven in \cite{computability}.
\begin{theorem}[Parameterization Theorem]\label{thm:param}
Given a 2-place partial recursive function $\psi(w,x)$, we can find a 1-place recursive function $h(w)$ such that
$$\varphi_{h(w)}(x) \simeq \psi(w,x) $$
for all $w,x$.
\end{theorem}

A set $A \subseteq 2^\nn$ is called $\sz{1}$ if, for some recursive predicate $R$
$$ X \in A \equiv \exists \; n \in \nn \, :  R(X,n) \mbox{ holds}.$$

A set $A$ is $\pz{1}$ if, for some recursive predicate $R$
$$ X \in A \equiv \forall \; n \in \nn \, :  R(X,n) \mbox{ holds}.$$

We will use the standard enumeration of $\pz{1}$ and $\sz{1}$ sets in the Cantor space, given by
$$U_{e} = \{ X \in 2^\nn : \varphi_{e}^{X}(0) \uparrow \} \qquad S_{e} = \{X \in 2^\nn : \varphi_{e}^{X}(0) \downarrow \} $$
respectively, where $\uparrow$ means that the expression is undefined and $\downarrow$ means that it is defined.

We say a sequence $V_n$ of $\sz{1}$ sets is \emph{effectively open} or \emph{uniformly $\sz{1}$} if $V_n = S_{f(n)}$, where $f$ is a total recursive function. A set $S$ is \emph{effectively null} if $S \subseteq \cap_{n=0}^{\infty} V_n$, where $V_n$ are uniformly $\sz{1}$ and $\mu(V_n) \leq \frac{1}{2^n}$.


\subsection{Definitions of Randomness} 
\label{ssub:definitions_of_randomness}
There are 3 definitions of randomness that will be considered in this paper.

A point $X \in 2^\nn$ is called \emph{weakly random} if it does not belong to any $\pz{1}$ set of measure 0.

A point $X \in 2^\nn$ is said to be \emph{random} (in the sense of Martin-L\"of) if it does not lie in any effectively null set. Equivalently, the singleton set $\{ X\}$ is not effectively null.

A point $X \in 2^\nn$ is called \emph{strongly random} if it does not belong to any $\pz{2}$ set of measure 0.

The names assigned to these different types of randomness are justified, as the following shows. The result follows from the fact that all null $\pz{1}$ sets are effectively null and all effectively null sets are contained in $\pz{2}$ null sets.

\begin{theorem}\label{thm:randomheir} Let $X \in 2^\nn.$
\begin{enumerate}
\item $X$ is random $\Rightarrow X$ is weakly random.
\item $X$ is strongly random $\Rightarrow X$ is random.
\end{enumerate}
\end{theorem}

This tells us that (Martin-L\"of) randomness is intermediate between weak randomness and strong randomness. As will be shown later, weak randomness is genuinely weaker than randomness, as there are certain statistical properties that random elements have that weakly random elements do not.

Similarly, strong randomness is stronger than randomness, as there are relations involving Turing reducibility that hold for strongly random elements that do not hold for random elements.  For example, we have that if $a$ represents the Turing degree of a strongly random sequence $X$ and $0'$ represents the Turing degree of the halting problem, then $\inf(a, 0')$ is recursive. The same result does not hold with the hpyothesis that $X$ be random, however.

One of the important consequences of the Martin-L\"of definition of randomness is that a computable-theoretic analogue of the Borel-Cantelli lemma from probability theory holds, giving us a powerful tool for proving statements about random sequences. A proof is given in \cite{davie} (stronger than what will be stated here), cast in the light of complexity theory, but we present proofs that are more consistent with our measure-theoretic approach to randomness.

\begin{lemma}[Solovay's Lemma]\label{lem:solo}
Let $V_1, V_2, \dots$ be a sequence of uniformly $\sz{1}$ sets, then:
\begin{enumerate}
\item If $X$ is random and $\sum_{n=1}^{\infty} \mu(V_n) < \infty$, then $X$ lies in only finitely many $V_n$.
\item If $X$ is weakly random, $\sum_{n=1}^{\infty}\mu(V_n)$ diverges, and $V_n$ represent mutually independent events, then for each $m$ there exists an $N$ such that for some $m \leq i \leq N$, $X \in V_i$.
\end{enumerate}
\end{lemma}

A proof of $(1)$ is contained in, so we will prove (2).
\begin{proof}[Proof of (2)]
Choose an $m$ and consider the set
\[ Q_m = \bigcup_{n=m}^\infty V_n.\]
We have that $Q_m$, being the union of $\Sigma_1^0$ sets, is $\Sigma_1^0$. By De Morgan's Laws and the fact that $V_n$ are mutually independent, we have that
\[ \overline{Q_m} = \bigcap_{n=m}^\infty \overline{V_n}\]
and
\[ \mu(\overline{Q_m} ) = \prod_{n=m}^\infty (1 - \mu(V_n)).\]
We then will have, noting the above and the fact that the sum of the measures of $V_n$ diverges, $\mu(Q_m) = 1$ for all $m$. This is because (as in Lemma 5.11 of \cite{simpson2007almost}), for a sequence $a_i \in (0,1)$,
\[ \sum_{m=n}^\infty a_m = \infty \quad \text{if and only if} \quad \prod_{m=n}^\infty (1-a_m) = 0. \]
Letting $a_i = \mu(V_n),$ we see that $\mu(\overline{Q_n}) = 0$ and so $\mu(Q_n) = 1$. For $X$ to be weakly random, it must lie in $Q_m$ for all $m$ (the complement of $Q_m$ being a $\Pi_1^0$ set of measure 0), and therefore $X$ is in $V_n$ for infinitely many $n$.
\end{proof}

\subsection{Results From Probability Theory} 
\label{ssub:results_from_probability_theory}
In order to prove some later results, we will need some results from probability theory, taken from \cite{feller}. We state them in the case where the probability space in question is $2^\nn$ with $\mu$ defined as above. $\mathbb{E}$ will refer to the expectation of a random variable with respect to $\mu.$ Finite sequences of coin tosses correspond to initial segments of sequences in $2^\nn$.

Given $X \in 2^\nn$ with $S_n = \sum_{i=1}^{n} X_i$, define the \emph{reduced number of heads in $n$ tosses} by the quantity
$$S_n^* = \frac{S_n - \frac{n}{2}}{\sqrt{\frac{n}{4}}}.$$
\begin{lemma}\label{prlem}
Let $x$ be fixed and let $A$ be the event that for at least one $k$ with $k \leq n$
$$S_k - \frac{k}{2} > x.$$
Then there exists a constant $c$, independent of $x$, such that for all $n$
$$\mu(A) \leq \frac{1}{c} \mu\left(S_n - \frac{n}{2}> x\right).$$
\end{lemma}
\begin{theorem}\label{thm:deviations}
If $n \rightarrow \infty$, $x \rightarrow \infty$ and $x^3 = o(\sqrt{npq})$, then
$$\mu\left(S_n^* > x\right) \sim \frac{e^{-\frac{1}{2}x^2}}{\sqrt{2 \pi} x}. $$
\end{theorem}

\begin{lemma}[Hoeffding's Inequality]\label{lem:hoeffding}
Given a sequence of i.i.d. random variables $X_i$ on $2^\nn$ that take on values in $[a,b]$ almost surely, for their sum $S_n = X_1 + \dotsm + X_n$ and $\epsilon > 0 $ we have
$$\mu \left ( \left \{ \omega \in 2^\nn : \frac{ \lvert S_n- \mathbb{E}(S_n)  \rvert}{n} \geq \epsilon\right \} \right) \leq 2\exp\left (\frac{-2 n^2 \epsilon^2
}{(b-a)^2}\right) $$
\end{lemma}

We also state a version of Hoeffding's Inequality to be used specifically for the proof of \cref{thm:SLLN}.

\begin{lemma}[Hoeffding's Inequality - Fair Coin Case]\label{lem:faircoinhoeffding} For $X \in 2^\nn,$, all $\epsilon > 0$ and $S_n = \sum_{i=0}^{n-1} X(i)$, we have that
\[\mu \left ( \left \lvert \frac{S_n}{n} - \frac{1}{2} \right \rvert > \epsilon \right ) < 2 \exp(-2n\epsilon^2).\]
\end{lemma}


\section{Statistical Properties} 
\label{sec:statistical_properties}
\subsection{The Strong Law of Large Numbers} 
\label{sub:the_strong_law_of_large_numbers}
The Strong Law of Large Numbers (SLLN) is a basic result about i.i.d. random variables that describes the deviation of their average values from their common mean. We state it as applied to the Cantor space for the standard example where we take the random variables to be the values of a sequence.

\begin{theorem}[Strong Law of Large Numbers: Standard Form]\label{thm:clSLLN}
Choose an $X \in 2^\nn$, then for the partial sums $S_n = \sum_{i = 1}^{n} X_i$ we have
$$\mu \left[\frac{S_n}{n} \rightarrow \frac{1}{2}\right] = 1 .$$
\end{theorem}

We would, of course, hope that a random element of $2^\nn$ would correspond to the ubiquitous set of measure 1 on which i.i.d. random variables obey statistical laws such as the SLLN. Using the Martin-L\"of definition, this is the case, as the following shows.

\begin{theorem}[Strong Law of Large Numbers: Computable Form]\label{thm:SLLN}
Suppose that $X \in 2^\nn$ is random. Then the values of $X$ obey the Strong Law of Large Numbers. That is, for the sum $S_n$
\[ \frac{S_n}{n} \rightarrow \frac{1}{2}.\]
\end{theorem}

\begin{proof}\label{pf:SLLN}
Suppose that $X$ does not obey the SLLN. Then we have that $\exists m \; \forall N \; \exists n$ such that \[ \left| \frac{\sum X(i)}{n} - \frac{1}{2}\right| > \frac{1}{m}.\] Consider the set $V_m$, which $X$ will belong to:
$$V_m = \left \{\omega \in 2^{\nn} : \;\forall N \exists (n > N) \ \biggl \lvert \frac{\sum \omega(i)}{n} - \frac{1}{2} \biggr \rvert > \frac{1}{m} \right\}. $$
Note that $V_m = \bigcap_{N=0}^{\infty} V_{m,N} = \bigcap_{N=0}^\infty \bigcup_{K >N}^\infty Q_K$, where
\begin{eqnarray*}
&V_{m,N} &= \left \{\omega \in 2^{\nn} : \exists (K > N) \ \left \lvert \frac{\sum \omega(i)}{n} - \frac{1}{2} \right  \rvert > \frac{1}{m}\right \} \\
&	&=\bigcup_{K>N}^{\infty} \left \{ \omega \in 2^{\nn} : \left \lvert \frac{\sum \omega(i)}{K} - \frac{1}{2} \right \rvert > \frac{1}{m}  \right \} \\
& &= \bigcup_{K>N}^{\infty} Q_K \\
\end{eqnarray*}
We have that $V_{m,N}$ is a uniform sequence of $\Sigma^{0}_1$ sets, as $V_{m,N}$ is the domain of the partial recursive functional \[P(x,N,k) \simeq \text{the least } \,n > N :  \left \lvert \frac{\sum \omega(i)}{n} - \frac{1}{2} \right \rvert > \frac{1}{m} .\] By the Parameterization Theorem (\cref{thm:param}), there is a total recursive function $h(N)$ such that \[P(\omega,N,k) = \varphi^{(1), \omega}_{h(N)}(k) = \varphi^{(1), \omega}_{h(N)}(0),\] so $V_{m,N} = U_{h(N)}$, and $V_{m,N}$ is uniform.

By Hoeffding's Inequality (\cref{lem:faircoinhoeffding})) and $\sigma-$subadditivity of $\mu$, we have that:
\begin{eqnarray*}
&\mu(V_{m,N}) &\leq \sum_{K=N}^{\infty} \mu(Q_K) < \sum_{K=N}^{\infty} 2e^{\frac{-2K}{m^2}}  \\
& &= 2e^{\frac{-2N}{m^2}}(1+e^{\frac{-2}{m^2}}+e^{\frac{-4}{m^2}}+\dotsm) \\
& &= 2e^{\frac{-2N}{m^2}}\left (\frac{1}{1 - e^{\frac{-2}{m^2}}} \right) \\
\end{eqnarray*}

Because $\lvert e^{\frac{-2}{m^2}} \rvert < 1$, this series will converge, and we have an upper bound for $\mu(V_{m,N})$. To show that $V_{m}$ is an effectively null set, we choose a subsequence of $V_{m,N}$, $V_{m,N,k}$ such that $\mu(V_{m,N_k}) \leq \frac{1}{2^k} $. This can be done in a computable fashion, as we have a computable upper bound on $\mu(V_{m,N})$ as a function of $N$.

We then have that $V_{m} = \cap_{k = 0}^\infty V_{m,N_k}$ is effectively null. Recall that $X$ was supposed to be random, and as such cannot lie in $V_{m}$. This is a contradiction, and so $X$ cannot belong in $V_{m}$ for any $m$ and therefore must obey the SLLN.
\end{proof}

\subsection{Normality of Random Elements} 
\label{sub:normality_of_random_elements}
One property that we would like random elements of $2^\nn$ to have is that their digits should be \emph{normal} (as binary numbers). What this means is that, given a finite sequence of 0's and 1's, the likelihood of this sequence appearing in the values of the random element should be consistent with a uniform distribution. The notion of normality was introduced in a 1909 paper by Borel \cite{emile1909probabilites}, where he proved that almost all numbers were normal \footnote{An explanation of normal numbers, as well as a proof of Borel's result, can be found in the article \cite{khoshnevisan2006normal}}.

To state the definition precisely, we say that $X \in 2^\nn $is \emph{normal} if for all $\sigma \in 2^l$, the limit as $n \rightarrow \infty$ of the fraction of $i$'s less than $n$ such that $\langle X(i), X(i+1), \dots, X(i+l-1)\rangle = \sigma$ is $\frac{1}{2^l}$.  That is, the probability of an arbitrary bitstring of length $l$ appearing in its values has limiting probability $\frac{1}{2^l}$.

Note that this is a \emph{stronger} statement than the Strong Law of Large Numbers, which is the special case in which we only consider bitstrings of length 1 for a particular sequence of i.i.d's. The following theorem establishes that Martin-L\"of random elements of $2^\nn$ share this property.
\begin{theorem}
Suppose that $X \in 2^{\mathbb{N}}$ is random, then the values of $X$ are normal.
\end{theorem}
\begin{proof}Letting $\sigma \in 2^{k}$, define functions $M_i^{(k)} : 2^{\nn} \rightarrow \rr$ as follows:
\begin{equation*}
M_i ^{(k)}(\omega) =
\begin{cases}
1 & \mbox{if $\sigma$ occurs starting at $\omega(i)$} \\
0 & \mbox {otherwise} \\
\end{cases}
\end{equation*}
Considering these functions as random variables on the Cantor space, we will have that $M_{ki}^{(k)}$ will be independent and identically distributed for $i = 0, 1, \dots$. We also have that  $\mathbb{E}(M_{ki}^{(k)})=  \frac{1}{2^k}$ for any choice of $\sigma$.

\noindent The $M_{ki}^{(k)}$ we have defined are i.i.d.'s with common expectation $\frac{1}{2^k}$ and take on values in $[0,1]$, so by Lemma \ref{lem:hoeffding}, we have

$$
\mu\left (\left \{\omega \in 2^{\nn} : \biggl \rvert\frac{ \sum_{i=0}^{n-1} M_{ki}^{(k)}(\omega)}{n} - \frac{1}{2^k} \biggl \lvert  \geq \epsilon\right \}\right) \leq 2\exp(-2n\epsilon^2)
$$
Choose $\epsilon >0 \in \mathbb{Q}$ and consider the sets $V_{n}^{(k)}$, defined as follows:
$$
V_{n}^{(k)} = \left\{ \omega \in 2^{\nn} : \biggl \rvert\frac{ \sum_{i=0}^{n-1} M_{ki}^{(k)}(\omega)}{n} - \frac{1}{2^k} \biggl \lvert  > \epsilon\right\}
$$
By above, we have that $\mu(V_n^{(k)}) \leq 2\exp(-2n\epsilon^2)$, and therefore
$$
0 \leq \sum_{n=1}^{\infty} \mu(V_n^{(k)}) \leq \sum_{n=1}^{\infty} 2 \exp(-2n\epsilon^2)  < \infty
$$
We also have that $V_n^{(k)}$ are uniformly $\sz{1}$ (by a similar argument as in the proof of the SLLN), so by Solovay's Lemma, we have that a random point $X$ can lie in only finitely many of the $V_{n}^{(k)}$. As $\epsilon$ was arbitrary (in $\mathbb{Q}$), we then have that
$$
\forall \epsilon >0 \in \mathbb{Q} \quad \exists N \;\forall n \geq N : \, \biggl \rvert\frac{ \sum_{i=0}^{n-1} M_{ki}^{(k)}(X)}{n} - \frac{1}{2^k} \biggl \lvert < \epsilon
$$
This means exactly that
$$
\lim_{n\rightarrow \infty} \frac{ \sum_{i=0}^{n-1} M_{ki}^{(k)}(X)}{n} = \frac{1}{2^{k}}
$$
which says that blocks of the string $\sigma$, in the limit, appear in the values of $X$ with probability $\frac{1}{2^{k}}$.

Because $\sigma$ was arbitrary, this tells us that the probability of an arbitrary string starting at multiples of its length in the values of $X$ is consistent with a normal distribution. However, we need that the string occurs with the same probability at all \emph{offsets}, not just in blocks. This is easily shown, though, as we can consider random variables of the form $M_{ik + n}^{(k)}$, for $n = 0, 1, \dots, k-1$. We can perform the same argument for these random variables as we did for the original ones, and we will arrive at the conclusion that the probability of an arbitrary string starting at any offset in the values of $X$ is also consistent with the expected value, telling us that $X$ is normal.
\end{proof}
When we consider elements of $2^\nn$ that are just weakly random, however, we do not have the same result. In fact, we do not even have that weakly random numbers obey the SLLN which, as was mentioned, is a weaker condition than normality.
\begin{lemma}
There exist $X \in 2^\nn$ that are weakly random and do not obey the Strong Law of Large Numbers.
\end{lemma}

\begin{remark}
The technique that will be used to construct a counterexample is that of \emph{finite approximation.} The idea is to describe the bitstrings that are initial segments of an element of $2^\nn$ inductively, fulfilling some condition at every step and letting the element be the union of the bitstrings described in the construction process.
\end{remark}

\begin{proof}
We proceed by finite approximation, constructing a sequence that does not belong to any null $\Pi^{0}_1$ set and also has its values weighted so that their average does not approach $\frac{1}{2}$.

\textbf{Stage 0:} Let $\sigma_0 = <>$.

\textbf{Stage $2n$:} \emph{Case 1:} Check if there exists $\sigma \supset \sigma_{n-1}$
such that $\varphi_n^{(1),\sigma}(0) \downarrow$. In this case, let $\sigma_n = \sigma$.

\emph{Case 2:} Not Case 1. Let $\sigma_n = \sigma$.

\textbf{Stage $2n+1$:} \emph{Case 1:} Check if $\frac{\sum_{i=0}^{|\sigma_{n}|-1} \sigma_{n}(i)}{|\sigma_n|} < \frac{3}{4} $. If so, then extend $\sigma_n$ with a string of 1's long enough to make it so that $\frac{\sum_{i=0}^{|\sigma_{n}|-1} \sigma_{n}(i)}{|\sigma_n|} \geq \frac{3}{4}$.

\emph{Case 2:} Not Case 1. Let $\sigma_{n+1} = \sigma_n^\smallfrown <1>$.

Let $X = \cup_{n=0}^{\infty} \sigma_n$. We then have that $X$ does not obey the SLLN by construction, and is also weakly random. To see why $X$ is weakly random, consider what happened at stage $n$ of the construction. If an extension existed that made $\varphi_{n}^{(1),\sigma}$ defined, we used it, and so $X$ will not belong to the $\Pi^{0}_1$ set determined by the index $n, U_{n}$. If such an extension did not exist, then we must know that $U_{n}$ is not a null set, as it is undefined on some neighborhood. We then have that $X$ does not belong to any $\Pi_{1}^{0}$ null set and is therefore weakly random.
\end{proof}
\begin{corollary}
There exist $X \in 2^\nn$ that are weakly random and whose digits are not normal.
\end{corollary}


\subsection{The Law of the Iterated Logarithm} 
\label{sub:the_law_of_the_iterated_logarithm}
In this section we will prove the Law of the Iterated Logarithm, which gives a convergence rate for the average value of a sequence of i.i.d. random variables. We first state it in its standard form as it applies to the Cantor space.

\begin{theorem}[Law of the Iterated Logarithm: Standard Form]\label{thm:LILst}
Let $X \in 2^\nn$ then for the sum $S_n$ we have we have
$$\limsup_{n\rightarrow \infty} \frac{|S_n- \frac{n}{2}|}{ \sqrt{\frac{n}{2} \log \log n}} = 1 $$
with probability 1.
\end{theorem}

This is clearly a strong statistical property that we would hope would follow from a proper definition of randomness. In the case where we are using Martin-L\"of randomness, this is true, as the following shows. The proof follows a standard approach in probability theory (taken from \cite{feller}), adapted to a computablility-theoretic context (a sketch of which is contained in \cite{davie}).

\begin{theorem}[Law of the Iterated Logarithm: Effective Form] Suppose $X \in 2^\nn$ is random. Letting $S_n = \sum_{i=0}^{n-1} X(i)$, we have that
$$ \limsup_{n\rightarrow \infty} \frac{S_n - \frac{n}{2}}{\sqrt{\frac{n}{2}\log \log n}} = 1.$$ That is:
\begin{enumerate}
\item For all $\lambda > 1$ there exists an $N$ such that for all $n > N$
\begin{equation}
\label{eq:lil1}
S_n \leq \frac{n}{2} + \lambda \sqrt{\frac{n}{2}\log \log n}.
\end{equation}
\item For all $\lambda < 1$  there exist infinitely many $n $ such that
\begin{equation}\label{eq:lil2}
S_n > \frac{n}{2} + \lambda\sqrt{\frac{n}{2}\log \log n}.
\end{equation}
\end{enumerate}
We have that if $X$ is random, then $\neg X$ is random ($\neg X$ obtained by flipping all the values of $X$). Then, by symmetry, the theorem implies that
\begin{equation}
\liminf_{n\rightarrow \infty} \frac{S_n - \frac{n}{2}}{\sqrt{\frac{n}{2}\log \log n}} = -1
\end{equation}
and therefore
\begin{equation}
\limsup_{n\rightarrow \infty} \frac{|S_n - \frac{n}{2}|}{\sqrt{\frac{n}{2}\log \log n}} = 1.
\end{equation}
\end{theorem}
\begin{note}
In the following proof we have to take care to only deal with computable numbers, but this is not a worry as computable numbers are dense and so we can choose an arbitrarily good computable approximation to a non-computable number if we are given one.
\end{note}
\begin{proof}[Proof of (1)]
Fix a computable $\lambda > 1$, let $\gamma$ be a computable number between 1 and $\lambda$ and let $n_r$ be the integer nearest to $\gamma^r$. Define $S_n$ as before and let $B_r$ be the set
$$
B_r = \left \{ X \in 2^\nn : S_n - \frac{n}{2} > \lambda \sqrt{\frac{n_r}{2}\log \log n_r} \mbox{ for some $n_r \leq n \leq n_{r+1}$}\right\}.
$$
$B_r$ is clearly a uniform sequence of $\sz{1}$ sets, and by Solovay's Lemma it suffices that $\sum_{r=0}^{\infty} \mu(B_r)$ converges in order to prove the claim.
We have that by \cref{prlem}:
\begin{align*}
\mu(B_r) &\leq \frac{1}{c} \mu\left (\left \{X \in 2^\nn : S_{n_{r+1}} - \frac{n_{r+1}}{2} > \lambda \sqrt{\frac{n_r}{2} \log \log n_r}\right\}\right) \\
&=  \frac{1}{c} \mu\left (\left \{ X \in 2^\nn :\frac{ S_{n_{r+1}} - \frac{n_{r+1}}{2}}{\sqrt{\frac{n_{r+1}}{4}}} > \lambda \sqrt{2 \frac{n_r}{n_{r+1}} \log \log n_r} \right \} \right)\\
&= \frac{1}{c} \mu\left (\left \{ X \in 2^\nn : S_{n_{r+1}}^* > \lambda \sqrt{2 \frac{n_r}{n_{r+1}} \log \log n_r} \right\}\right)
\end{align*}
Because $\frac{n_{r+1}}{n_r} \sim \gamma < \lambda$, we can choose $r$ large enough so that
$$
\mu(B_r) \leq \frac{1}{c} \mu \left (\left \{ X \in 2^\nn : S_{n_{r+1}}^* > \sqrt{2 \lambda \log \log n_r} \right \}\right).
$$
And so, by \cref{thm:deviations} we can choose $r$ large enough so that
$$
\mu(B_r) \leq \frac{1}{c} \exp(-\lambda \log \log n_r) = \frac{1}{c(\log n_r)^\lambda}
$$
We have that $\frac{1}{c(\log n_r)^\lambda}$ is arbitrarily close to $\frac{1}{c(r \log \gamma)^\lambda}$ and so for any given $k$ we can choose an $s$ such that $\sum_{r=s}^\infty \mu(B_r) < 2^{-k}$ and so $\sum_{r=0}^\infty \mu(B_r)$ not only converges, but is a computable number. The result follows, then, from the first part of Solovay's Lemma.
\end{proof}

\begin{proof}[Proof of (2)]
As before, let $\lambda < 1$ be computable and choose a computable $\eta > \lambda$ such that
\begin{equation}\label{eq:etadef}
1 -\eta < \left(\frac{\eta - \lambda}{2}\right)^2
\end{equation}
and choose $\gamma \in \nn$ such that $\frac{\gamma - 1}{\gamma} > \eta > \lambda.$ Set $n_r = \gamma^r$ and let
\begin{equation}\label{eq:dr}
D_r = S_{n_r} - S_{n_{r-1}}
\end{equation}
We have that $D_r$ represents the value of $S_n$, but only considered in blocks, so $\frac{D_r}{n_r - n_{r-1}}$ represents the average value of $S_n$ in the block of trials between $n_{r-1}$ and $n_r$. It is clear that the following sets, $A_r$, will correspond to independent events (from a probabilistic standpoint), as what occurs in $D_r$ is not affected by $D_k$ for $k\neq r$.
\begin{equation}\label{eq:ardef}
A_r = \{ X \in 2^\nn : D_r - \frac{n_r - n_{r-1}}{2} > \eta \sqrt{\frac{n_r}{2}\log \log n_r }\} .
\end{equation}
We have that
$$
\mu(A_r) = \mu\left( \left\{X \in 2^\nn : \frac{D_r - (n_r - n_{r-1}) \frac{1}{2}}{\sqrt{(n_r - n_{r-1})\frac{1}{4}}} > \eta \sqrt{2 \frac{n_r}{n_r - n_{r-1}} \log \log n_r} \right \}\right).
$$
Note that $n_r/(n_r - n_{r-1}) = \gamma / (\gamma - 1) < \eta^{-1}$, so
\begin{align*}
\mu(A_r) &\geq \mu\left(\left\{ X \in 2^\nn : \frac{D_r - (n_r - n_{r-1}) \frac{1}{2}}{\sqrt{(n_r - n_{r-1}) \frac{1}{4}}} > \sqrt{2 \eta \log \log n_r} \right\}\right) \\
&= \mu\left(\left\{ X \in 2^\nn : D_r^* > \sqrt{2 \eta \log \log n_r} \right\}\right)
\end{align*}
Again, by \cref{thm:deviations} we can choose $r$ great enough so that
$$
\mu(A_r) > \frac{1}{\log \log n_r} \exp(-\eta \log \log n_r) = \frac{1}{(\log \log n_r)(\log n_r)^\eta}.
$$
Because $n_r = \gamma^r$ and $\eta < 1$, we can choose $r$ large enough so that $\mu(A_r) > \frac{1}{r}$ and so the sum will diverge. From the second part of Solovay's Lemma, then, we have that the theorem holds when we consider $D_n$ instead of $S_n$.

It then remains to show that the $S_{n_{r-1}}$ term in \eqref{eq:dr} can be neglected, so that we can get the right form of the law.

From the first part of the theorem, (choosing $\lambda = 2$) we have that there is an $N$ such that for $r > N$ and all random $X$
\begin{equation}\label{eq:fix}
\left|S_{n_{r-1}} - \frac{n_{r-1}}{2 }\right| < 2\sqrt{\frac{n_{r-1}}{2} \log \log n_{r-1}}.
\end{equation}
Choose $\eta$ as in \eqref{eq:etadef}, and then we will have
$$
n_{r-1} = \frac{n_r}{\gamma} < n_r(\frac{\eta - \lambda}{2})^2
$$
and so by \eqref{eq:fix}
\begin{equation}\label{eq:grfix}
S_{n_{r-1}} - \frac{n_{r-1}}{2} > -(\eta - \lambda)\sqrt{\frac{n_{r}}{2}\log \log n_r}
\end{equation}
Adding \eqref{eq:grfix} to the condition on the set in \eqref{eq:ardef} (which we know holds infinitely often for all random $X$) we obtain
\begin{align*}
S_{n_r} - S_{n_{r-1}} - \frac{n_r - n_{r-1}}{2} + S_{n_{r-1}} - \frac{n_{r-1}}{2} >& (\eta -(\eta - \lambda) )\sqrt{\frac{n_r}{2}\log \log n_r } \\
S_{n_r} - \frac{n_r}{2} >& \lambda \sqrt{\frac{n_r}{2}\log \log n_r }
\end{align*}
which is \eqref{eq:lil2} with $n = n_r$. It follows that for all random $X$ the inequality holds for infinitely many $r$ and so the second part of the theorem is established.
\end{proof}

\section{Concluding Remarks}
\label{sec:conclusion}
As the above results show, there are several powerful advantages of the Martin-L\"of definition of randomness.

Firstly, we have Solovay's Lemma which enables us to prove a great many results from probability theory that generally hold for i.i.d. random variables. The fact that a proof from probability theory can be copied nearly verbatim (as in the proof of the Law of the Iterated Logarithm above) and applied to Martin-L\"of randomness is a convincing argument for its usefulness. Because of this correspondence, it seems likely that other results from probability theory such as the Central Limit Theorem could be adapted to a computability-theoretic context and proven.

Secondly, we have that, by a theorem of Schnorr (see \cite{simpson2007almost}), the Martin-L\"of definition is precisely equivalent to requiring that random elements be ``incompressible" in terms of their informational content. This is not only an intuitively satisfying property but a useful one, giving us more tools to work with random elements. While it was not discussed here, it is good to note that there are other such equivalencies involving randomness in the context of LR/LK reducibility.

Finally, we see that Martin-L\"of randomness is (at least within the hierarchy of arithmetical randomness we presented) the weakest randomness that ``gets the job done'' of satisfying our intuitive requirements of randomness. We would, of course, like to use the weakest definition necessary to do this. As we have seen, weakly random elements do not have the nice properties that random elements do, while strongly random elements \emph{will}. This shows us that weak randomness is insufficient, while strong randomness is, in a sense, overkill.

It must be admitted that there \emph{are} some properties of Martin-L\"of random numbers that do not coincide with the intuitive expectations of randomness. For example, we can exhibit random sequences that are low (their Turing jumps are equivalent to the Halting Problem) as well as random sequences that are hyperimmune-free (all sequences Turing reducible to the original sequence are bounded recursively). Despite this, though, the other properties of Martin-L\"of random numbers provide a strong argument for this definition of randomness being a good one.

\bibliography{bib}
\bibliographystyle{jloganal}

\end{document}